\documentclass[12pt]{article}
\usepackage{a4}
\usepackage{amsthm}
\usepackage{amsfonts}
\usepackage{amssymb}
\usepackage{stmaryrd}
\usepackage{cite}
\usepackage{epsfig}
\newtheorem{theorem}{Theorem}
\newtheorem{lemma}[theorem]{Lemma}
\newcommand{\conf}[1]{{\tt #1}}
\def\url{{\tt http://www.ucw.cz/\~{}kral/cyclic-six/}}
\begin{document}
\title{Third case of the Cyclic Coloring Conjecture~\thanks{An extended abstract containing this result has appeared in the proceedings of EuroComb'15 published in Electronic Notes in Discrete Mathematics. The work leading to this invention has received funding from the European Research Council under the European Union's Seventh Framework Programme (FP7/2007-2013)/ERC grant agreement no.~259385.}}
\author{Michael Hebdige\thanks{Department of Computer Science, University of Warwick, Coventry CV4 7AL, UK. E-mail: {\tt m.hebdige@warwick.ac.uk}.} \and
\and
        Daniel Kr\'al'\thanks{Mathematics Institute, DIMAP and Department of Computer Science, University of Warwick, Coventry CV4 7AL, UK. E-mail: {\tt d.kral@warwick.ac.uk}.}}
\date{}
\maketitle

\begin{abstract}
The Cyclic Coloring Conjecture asserts that
the vertices of every plane graph with maximum face size $\Delta^*$
can be colored using at most $\lfloor 3\Delta^*/2\rfloor$ colors in such a way that
no face is incident with two vertices of the same color.
The Cyclic Coloring Conjecture has been proven only for two values of $\Delta^*$:
the case $\Delta^*=3$ is equivalent to the Four Color Theorem and
the case $\Delta^*=4$ is equivalent to Borodin's Six Color Theorem,
which says that every graph that can be drawn in the plane with each edge crossed by at most one other edge is $6$-colorable.
We prove the case $\Delta^*=6$ of the conjecture.
\end{abstract}

\section{Introduction}

One of the most well-known open problems on coloring planar graphs is the Cyclic Coloring Conjecture,
which was made by Borodin in 1984~\cite{bib-borodin84} (the conjecture is sometimes thought
to have also been made by Ore and Plummer in the 1960's). 
The conjecture asserts that every plane graph with maximum face $\Delta^*$
has a cyclic coloring with at most $\lfloor 3\Delta^*/2\rfloor$ colors,
i.e. its vertices can be colored with at most $\lfloor 3\Delta^*/2\rfloor$ colors in such a way that
no two vertices incident with the same face get the same color.
The case $\Delta^*=3$ of the conjecture is equivalent to the Four Color Theorem,
which asserts that every planar graph is $4$-colorable and
which was proven in~\cite{bib-appel77,bib-appel77b}; a simpler proof was given in~\cite{bib-robertson99+}.
The only other known case of the conjecture is $\Delta^*=4$,
which is known as Borodin's Six Color Theorem~\cite{bib-borodin84,bib-borodin95}.
This case of the conjecture is equivalent to the following statement:
every graph embedded in the plane in such a way that each edge is crossed by at most one other edge is $6$-colorable.

There has been a substantial amount of work on the conjecture
both focused on proving upper bounds for particular values of $\Delta^*$,
which are summarized in Table~\ref{tab-bounds}, and on establishing general bounds. 
The work on general bounds~\cite{bib-borodin92,bib-ore69+,bib-borodin99+}
culminated with currently the best known general bound
$\lceil 5\Delta^*/3\rceil$ due to Sanders and Zhao~\cite{bib-sanders01+}. 
Amini, Esperet and van den Heuvel~\cite{bib-amini08+},
extending the work from~\cite{bib-havet07+,bib-havet08+}, proved that
the conjecture holds asymptotically in the following sense:
for every $\varepsilon>0$, there exists $\Delta_0$ such that
every plane graph with maximum face size $\Delta^*\ge\Delta_0$
has a cyclic coloring with at most $\left(\frac{3}{2}+\varepsilon\right)\Delta^*$ colors.

\begin{table}
\begin{center}
\begin{tabular}{|l|c|c|c|c|c|c|c|c|}
\hline
Value of $\Delta^*$ & 3 & 4 & 5 & 6 & 7 & 8 & 9 & 10 \\
\hline
Upper bound & 4 & 6 & 8 & {\bf 9} & 11 & 13 & 15 & 17 \\
Source & \cite{bib-appel77,bib-appel77b,bib-robertson99+} & \cite{bib-borodin84,bib-borodin95} & 
       \cite{bib-borodin99+} & {\bf here} &   
       \cite{bib-havet+} & \cite{bib-zlamalova} & 
       \cite{bib-borodin92} & \cite{bib-sanders01+}\\ 
\hline
Conjecture & 4 & 6 & 7 & 9 & 10 & 12 & 13 & 15 \\
\hline
\end{tabular}
\end{center}
\caption{The known upper bounds for the Cyclic Coloring Conjecture.}
\label{tab-bounds}
\end{table}

There has been no new exact results on the conjecture for more than 30 years.
In this paper, we resolve another case of the conjecture, proving the following.

\begin{theorem}
\label{thm-main}
Every plane graph with maximum face size at most six
has a cyclic coloring using at most nine colors.
\end{theorem}

\noindent 
The proof of Theorem~\ref{thm-main} is based on a discharging argument involving 103 discharging rules and 193 reducible configurations.
Despite the high complexity of the argument,
we are able to present a proof of the reducibility of all configurations and
the analysis of the final amount of charge for all vertices and for all faces except those of sizes five and six,
where we had to resort to computer assisted techniques
to analyze the final amount of charge (Lemma~\ref{lm-fac56}).
We have prepared three different programs to verify the correctness of the proof of this lemma and
we have made one of the programs available at \url .
We have also uploaded its source code to arXiv as an ancillary file.

Before presenting the proof of our main result, we would like to mention two closely related conjectures;
additional related results can also be found in a recent survey by Borodin~\cite{bib-borodin13}.
One is the conjecture of Plummer and Toft~\cite{bib-plummer87+}, studied e.g.~in~\cite{bib-enomoto01+,bib-hornak99+,bib-hornak00+,bib-hornak10+},
asserting that every $3$-connected plane graph with maximum face size $\Delta^*$ has a cyclic coloring using at most $\Delta^*+2$ colors.
The other conjecture is the Facial Coloring Conjecture from~\cite{bib-kral05+},
which was studied e.g.~in~\cite{bib-havet10+,bib-havet+,bib-kral05+,bib-kral07+}.
This conjecture asserts for every positive integer $\ell$ that
every plane graph has an $\ell$-facial coloring with at most $3\ell+1$ colors,
i.e.~a vertex coloring such that any vertices joined by a facial walk of length at most $\ell$ receive different colors.
If the Facial Coloring Conjecture holds for a particular value of $\ell$,
then the Cyclic Coloring Conjecture holds for $\Delta^*=2\ell+1$.
Unfortunately, the only proven case of the Facial Coloring Conjecture is the case $\ell=1$, which is equivalent to the Four Color Theorem.
Still, partial results towards the proof of the Facial Coloring Conjecture
give the best known upper bound for the case $\Delta^*=7$ of the Cyclic Coloring Conjecture~\cite{bib-havet+}.

\section{Notation}

We follow the notation standard in the area of planar graph coloring.
All graphs considered in the following are plane graphs that could have parallel edges but do not have loops.
A vertex of degree $k$ is referred to as a $k$-vertex,
a vertex of degree at most $k$ as a $\le k$-vertex and
a vertex of degree at least $k$ as a $\ge k$-vertex.
The {\em degree of a face} is the number of vertices incident with it and
we use a $k$-face, a $\le k$-face and a $\ge k$-face in the analogous meanings.
Two vertices are {\em facially adjacent} if they are incident with the same face and
the {\em facial degree} of a vertex is the number of vertices facially adjacent to it.
In a $2$-connected plane graph,
each face is bounded by a cycle and proper connected subgraphs of this cycle are referred to as {\em facial walks}.
Finally, a cycle $C$ in a plane graph $G$ is {\em separating} if it does not bound a face neither inside nor outside.

When describing configurations in plane graphs, we will often describe $5$-faces and $6$-faces in the following way:
a $k$-face $v_1v_2\cdots v_k$, $k\in\{5,6\}$, will be represented by a string of length $2k+2$ characters starting with \conf{P:} or \conf{H:}
if $k=5$ or $k=6$, respectively.
The $(2i+1)$-th position will represent the type of the vertex $v_i$ and
the $(2i+2)$-th position will represent the type of the face sharing the edge $v_iv_{i+1}$ (indices modulo $k$).
The types of vertices and faces are encoded using the notation given in Tables~\ref{tab-ver} and~\ref{tab-fac}, respectively.
In both cases, we can use wildcards to represent several types of vertices and faces as given in Tables~\ref{tab-ver+} and~\ref{tab-fac+}.
Since a minimal counterexample to Theorem~\ref{thm-main} cannot contain a $3$-face and a $\le 5$-face sharing an edge,
we will consider configurations where every $3$-face shares edges with $6$-faces only.

\begin{table}
\begin{center}
\begin{tabular}{|c|p{12cm}|}
\hline
Type & Description \\
\hline
\conf{t} & a $3$-vertex such that its neighbor not on the face is a $\ge 4$-vertex\\
\conf{o} & a $3$-vertex such that its neighbor not on the face is also a $3$-vertex\\
\conf{v} & a $4$-vertex $v$ contained in a $3$-face $vv'v''$ such that
           neither $v'$ nor $v''$ is on the described face and
	   both $v'$ and $v''$ are $\ge 4$-vertices\\
\conf{u} & a $4$-vertex $v$ contained in a $3$-face $vv'v''$ such that
           neither $v'$ nor $v''$ is on the described face and
	   $v'$ and $v''$ are a $3$-vertex and $\ge 4$-vertex\\
\conf{w} & a $4$-vertex $v$ contained in a $3$-face $vv'v''$ such that
           neither $v'$ nor $v''$ is on the described face and
	   both $v'$ and $v''$ are $3$-vertices\\
\conf{4} & a $4$-vertex \\
\conf{5} & a $5$-vertex \\
\conf{6} & a $\ge 6$-vertex \\
\hline
\end{tabular}
\end{center}
\caption{The vertex type representation.}
\label{tab-ver}
\end{table}

\begin{table}
\begin{center}
\begin{tabular}{|c|p{12cm}|}
\hline
Type & Description \\
\hline
\conf{t} & a $3$-face $v_iv_{i+1}w$ such that $w$ is a $3$-vertex and its remaining neighbor is a $\ge 4$-vertex \\
\conf{O} & a $3$-face $v_iv_{i+1}w$ such that $w$ is a $3$-vertex and its remaining neighbor is a $3$-vertex \\
\conf{x} & a $3$-face $v_iv_{i+1}w$ such that $w$ is a $\ge 4$-vertex \\
\conf{Q} & a $4$-face\\
\conf{P} & a $5$-face\\
\conf{H} & a $6$-face\\
\hline
\end{tabular}
\end{center}
\caption{The face type representation assuming the faces share an edge $v_iv_{i+1}$.}
\label{tab-fac}
\end{table}

\begin{table}
\begin{center}
\begin{tabular}{|c|c|p{6cm}|}
\hline
Wildcard & Represented types & Description \\
\hline
\conf{3} & \conf{t} and \conf{o} & a $3$-vertex\\
\conf{x} & all but \conf{t} and \conf{o} & a $\ge 4$-vertex\\
\conf{+} & \conf{5} and \conf{6} & a $\ge 5$-vertex \\
\conf{*} & all & any type of a vertex\\
\hline
\end{tabular}
\end{center}
\caption{The vertex type wildcards.}
\label{tab-ver+}
\end{table}

\begin{table}
\begin{center}
\begin{tabular}{|c|c|l|}
\hline
Wildcard & Represented types & Description \\
\hline
\conf{3} & \conf{t} and \conf{O} & a $3$-face with the tip being a $3$-vertex \\
\conf{T} & \conf{t}, \conf{O} and \conf{x} & a $3$-face \\
\conf{F} & \conf{Q}, \conf{P} and \conf{H} & a $\ge 4$-face \\
\conf{*} & all & any type of a face\\
\hline
\end{tabular}
\end{center}
\caption{The face type wildcards.}
\label{tab-fac+}
\end{table}

The (most generic) $6$-face configuration described as \conf{H:3Q5*oO4Po*3*} and
the $5$-face configuration described as \conf{P:v*w*******} can be found in Figure~\ref{fig-example}.
When drawing faces, we will represent $3$-vertices with circles, $4$-vertices with squares and
$5$-vertices with pentagons (as shown in Figure~\ref{fig-example}).
Finally, if the description of a face ends with one or more stars, we often omit these stars.
In particular, the configurations depicted in Figure~\ref{fig-example}
can also be described as \conf{H:3Q5*oO4Po*3} and \conf{P:v*w}.

\begin{figure}
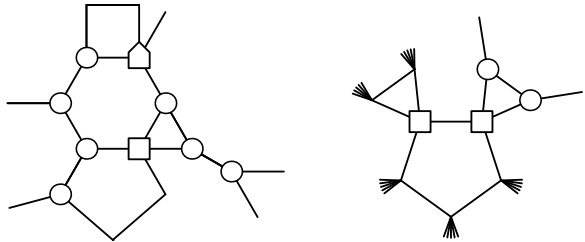

\begin{center}
\epsfbox{cyclic-six.1001}
\hskip 4ex
\epsfbox{cyclic-six.1002}
\end{center}
\caption{The most generic face configurations described by \conf{H:3Q5*oO4Po*3*} and by \conf{P:v*w*******}.}
\label{fig-example}
\end{figure}

\section{Overview of the proof}

We consider a minimal graph with maximum face size at most $6$ that has no cyclic coloring with at most $9$ colors;
the minimality is measured as the minimality of the sum of the numbers of vertices and edges.
Such a minimal graph is further referred to as a {\em minimal counterexample}.
It is easy to show that a minimal counterexample is $2$-connected, it has no parallel edges and its minimum facial degree is at least $9$.
In particular, the minimum degree of a minimal counterexample is at least $3$.
In addition, a minimal counterexample cannot contain a separating cycle of length at most $6$.
Also observe that a minimal counterexample does not contain a $3$-face that shares an edge with a $\le 5$-face.

We exclude the existence of a minimal counterexample (and thus prove Theorem~\ref{thm-main}) using the discharging method.
We fix a minimal counterexample and assign each $k$-vertex $k-4$ units of charge and each $k$-face $k-4$ units of charge.
Euler's formula implies that the sum of the amounts of the initial charges is $-8$.
We then apply the set of discharging rules described in Section~\ref{sec-rules}.
Based on these rules, some of the vertices and faces send charge to incident elements in such a way that the total sum of the charges is preserved.
However, we show that a minimal counterexample cannot contain any of the configurations described in Section~\ref{sec-reduc},
so-called reducible configurations, and using this we show that the final amount of charge of any vertex and any face is non-negative.
Since the amount of charge was preserved, this is impossible and hence excludes the existence of a counterexample
to the Cyclic Coloring Conjecture for $\Delta^*=6$, which finishes the proof.

\section{Reducible configurations}
\label{sec-reduc}

In this section, we identify configurations that cannot appear in a minimal counterexample.
To avoid an excessive use of wildcards, when we say that a certain configuration with the description containing \conf{v} is reducible,
we actually mean that the configurations with \conf{v} replaced with \conf{u} and \conf{w} are also reducible.
Likewise, the configurations with description containing \conf{u} are reducible with \conf{u} replaced with \conf{w}.
For example,
when we have established that the configuration \conf{P:v*3P3} is reducible (the configuration is depicted in Figure~\ref{fig-simple5}),
we have established that the configurations \conf{P:u*3P3} and \conf{P:w*3P3} are also reducible.

\subsection{Simple greedy reductions}
\label{sub-greedy}

The reducibility of most of the configurations will be established in the following way:
we consider a minimal counterexample $G$ containing the configuration,
possibly add some edges and
then contract one or more connected subgraphs to obtain a graph $G'$ with maximum face size at most six.
These subgraphs will be identified by the capital letters $A$, $B$, etc.~and
the resulting vertices of $G'$ will be denoted by $w_A$, $w_B$, etc.
If one or more loops appear because of the contraction, they get removed.

By the minimality of $G$, there exists a cyclic coloring of $G'$ using at most nine colors.
Most of the vertices of $G$ will keep the colors they are assigned in $G'$.
Two or more vertices of each subgraph $X=A,B,\ldots$ will get the color assigned to $w_X$ in $G'$
while the others remain uncolored.
The obtained coloring is then completed by coloring the non-colored vertices in a specific order.
This order is chosen in such a way that each vertex is facially adjacent to vertices with at most eight different colors
when it is supposed to be colored. Hence, the coloring can be completed to obtain a cyclic coloring of $G$.

Clearly, the vertices of the component $X$ that get the color of $w_X$ cannot be facially adjacent.
The next two lemmas will guarantee that certain pairs of the vertices of a subgraph $X$
are not facially adjacent.

\begin{lemma}
\label{lm-dist3}
If two vertices $u$ and $u'$ of a minimal counterexample $G$ are joined by a path of length $k\in\{2,3\}$ that is not a facial walk,
then $u$ and $u'$ are not facially adjacent.
\end{lemma}

\begin{proof}
Let $v_0\cdots v_k$ be the path between $u=v_0$ and $u'=v_k$.
Suppose that $u$ and $u'$ are facially adjacent.
Since the maximum face size of $G$ is at most six,
there is a facial walk $w_0\cdots w_{\ell}$ such that $u'=w_0$, $u=w_{\ell}$ and $\ell\in\{0,1,2,3\}$.
Since $v_0\cdots v_k$ is not a facial walk,
the closed walk $v_0\cdots v_kw_1\cdots w_{\ell-1}$ (note that $v_k=w_0$) contains a separating cycle of length at most $k+\ell\le 6$.
However, $G$ contains no separating cycle of length at most six.
\end{proof}

\begin{lemma}
\label{lm-dist4}
If two vertices $u$ and $u'$ of a minimal counterexample $G$ are joined by a path $v_0v_1v_2v_3v_4$, $u=v_0$ and $u'=v_4$,
such that $v_1v_2v_3$ is a facial walk and $v_2v_3v_4$ is a facial walk of another face,
then $u$ and $u'$ are not facially adjacent.
\end{lemma}

\begin{figure}
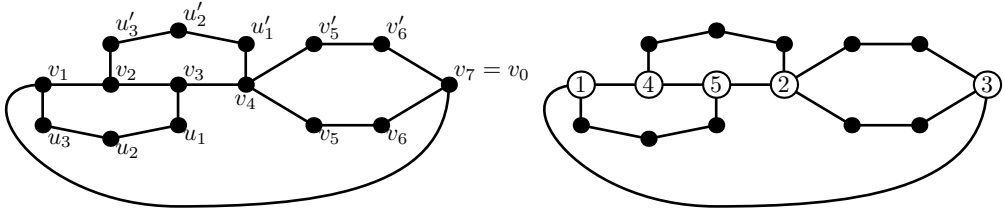

\begin{center}
\epsfbox{cyclic-six.1003}
\epsfbox{cyclic-six.1004}
\end{center}
\caption{Notation used in Lemma~\ref{lm-dist4}.}
\label{fig-dist4}
\end{figure}

\begin{proof}
If $u$ and $u'$ are facially adjacent, there is a facial walk $v_4\cdots v_k$ for $k\in\{4,5,6,7\}$ such that $v_k=u$ (and so $v_k=v_0$).
If $k\le 6$, then the closed walk $v_0v_1\cdots v_{k-1}$ would contain a separating cycle of length at most $k\le 6$, which is impossible.
Hence, we will assume that $k=7$ in the rest of the proof, i.e.~$u$ and $u'$ are the opposite vertices of a $6$-face.
Let $v_4v'_5v'_6v_7$ be the other facial walk between $u'$ and $u$ on this $6$-face,
let $v_1v_2v_3u_1\cdots u_{\ell}$ be the face containing the facial walk $v_1v_2v_3$ and
let $v_2v_3v_4u'_1\cdots u'_{\ell'}$ be the face containing the facial walk $v_2v_3v_4$.
By symmetry, we can assume that one side of the separating cycle $v_0v_1\cdots v_6$ contains
the face $v_1v_2v_3u_1\cdots u_{\ell}$ on one side and the face $v_2v_3v_4u'_1\cdots u'_{\ell'}$ and the vertices $v'_5$ and $v'_6$ on the other side.
See Figure~\ref{fig-dist4} for the illustration.

Let $H$ be the subgraph of $G$ on the side of the cycle $v_1v_2\cdots v_7$ with the face $v_1v_2v_3u_1\cdots u_{\ell}$
such that the path $v_1v_2v_3$ is replaced with the edge $v_1v_3$, and
let $H'$ be the subgraph of $G$ on the side of the cycle $v_1v_2v_3v_4v'_5v'_6v_7$ with the face $v_2v_3v_4u'_1\cdots u'_{\ell'}$
such that the path $v_2v_3v_4$ is replaced with the edge $v_2v_4$.
Note that the maximum face size of both $H$ and $H'$ is six.
The minimality of $G$ implies that both $H$ and $H'$ has facial colorings with at most nine colors.
The colors used by the two colorings are denoted by $1,2,\ldots,9$.

By symmetry, we can assume that the color of $v_1$ is $1$, that of $v_4$ is $2$ and that of $v_7$ is $3$ in both the colorings.
Moreover, we can assume that the color of $v_2$ in $H'$ is $4$ and that of $v_3$ in $H$ is $5$.
If the color of one of the vertices $v_5$ and $v_6$, say $v_i$, is different from the colors of $u_1,\ldots,u_{\ell}$,
we permute the colors of the vertices of $H$ in such a way that the color of $v_i$ is $4$,
the color of $v_{11-i}$ is $6$ and the colors of $u_1,\ldots,u_{\ell}$ are among $6,\ldots,9$.
If the color of both the vertices $v_5$ and $v_6$ appear among the colors of $u_1,\ldots,u_{\ell}$,
we permute the colors of the vertices of $H'$ in such a way that the colors of $v_5$ and $v_6$ are $6$ and $7$ and
the colors of $u_1,\ldots,u_{\ell}$ are among $6,7,8$.
We now permute the colors of the vertices of $H'$.
If the color of one of the vertices $v'_5$ and $v'_6$, say $v'_j$, is different from the colors of $u'_1,\ldots,u'_{\ell'}$,
we permute the colors of the vertices of $H$ in such a way that the color of $v'_j$ is $5$ and
the color of $v'_{11-j}$ is $8$ and the colors of $u'_1,\ldots,u'_{\ell'}$ are among $6,\ldots,9$.
If the color of both the vertices $v'_5$ and $v'_6$ appear among the colors of $u'_1,\ldots,u'_{\ell'}$,
we permute the colors of the vertices of $H'$ in such a way that the colors of $v'_5$ and $v'_6$ are $8$ and $9$ and
the colors of $u'_1,\ldots,u'_{\ell'}$ are among $6,8,9$.
It is easy to verify that the colorings of $H$ and $H'$ form a cyclic coloring of $G$.
\end{proof}

\begin{figure}
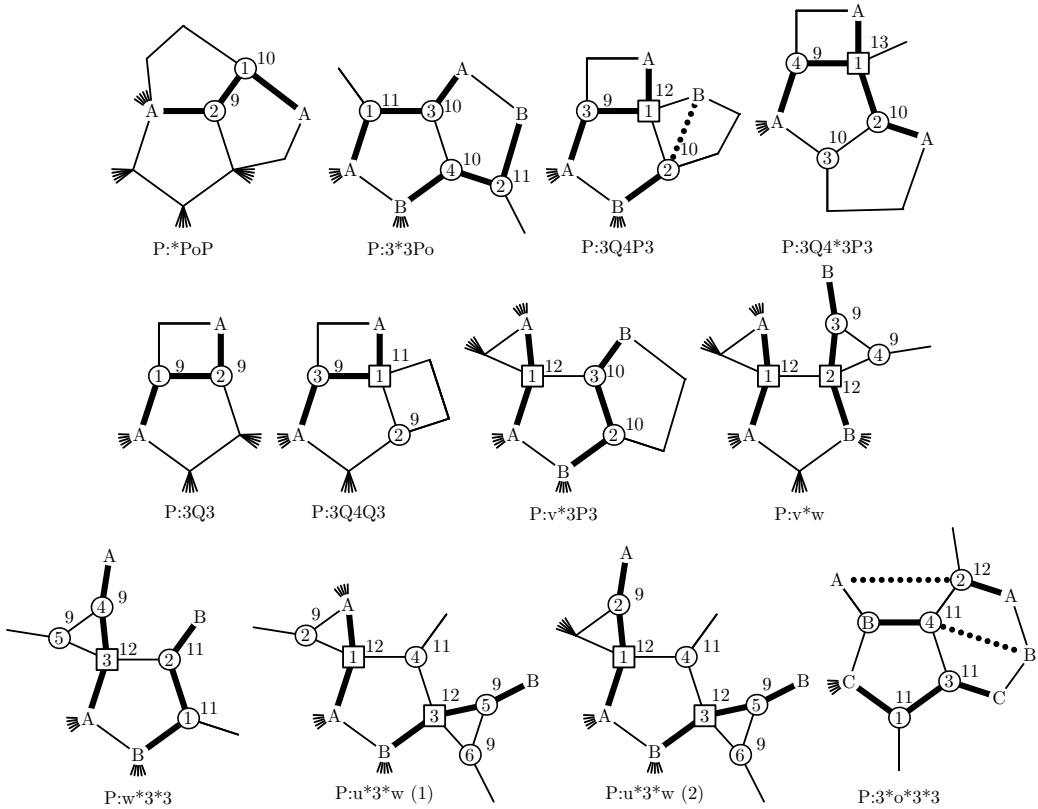

\begin{center}
\epsfbox{cyclic-six.1}
\epsfbox{cyclic-six.2}
\epsfbox{cyclic-six.3}
\epsfbox{cyclic-six.4}\\
\epsfbox{cyclic-six.5}
\epsfbox{cyclic-six.6}
\epsfbox{cyclic-six.7}
\epsfbox{cyclic-six.8}\\
\epsfbox{cyclic-six.9}
\epsfbox{cyclic-six.10}
\epsfbox{cyclic-six.11}
\epsfbox{cyclic-six.12}
\end{center}
\caption{The 11 reducible configurations related to 5-faces;
         note that there are two ways that the configuration with the description including \conf{u} can look like.
         All the depicted configurations are reducible in the simple greedy way.
         Also note that in the last configuration
         the vertices $w_A$, $w_B$ and $w_C$ in the reduced graph are incident with the same face and
	 thus they get distinct colors.}
\label{fig-simple5}
\end{figure}

Our proof uses 186 reducible configurations with their reducibility established in the way that we have just described.
The 11 such configurations related to 5-faces can be found in Figure~\ref{fig-simple5} and
the 175 configurations related to 6-faces in Figures~\ref{fig-simple6a}--\ref{fig-simple6i}.
In each of the configurations, the edges of the minimal counterexample that get contracted are depicted by bold and
the edges that are added and get contracted (if they exist) are bold and dotted.
The vertices that get the color of the vertex corresponding to the contracted component are marked by capital letters.
The numbered vertices are those that do not keep the colors and
their numbers give the order in that they are colored.
It is straightforward to verify that all the involved vertices are at distance at most six (and therefore they are distinct),
the vertices with the same capital letter satisfy the conditions of one of Lemmas~\ref{lm-dist3} and~\ref{lm-dist4}, and
each numbered vertex is facially adjacent to vertices with at most eight different colors when it gets a color.
To assist with the verification of the letter, the facial degrees of the vertices to be colored are displayed very near to them.

One more comment on the configurations depicted in Figures~\ref{fig-simple5}--\ref{fig-simple6i} is in place.
We always assume that the unconstrained faces around the considered face have size six.
Note that if their size is five or less and
this results in the absence of a vertex in one or more of the pairs $A$, $B$, etc., the counting argument for the greedy coloring would still work.
Instead of saving one color because of the facially adjacent pair of vertices with the same color,
we would save one color because the face size of the incident face is smaller.
Let us give an example.
If the face that is supposed to contain the vertices labelled with $A$, $B$, $C$, $3$, $4$ and $2$ (like in the last configuration in Figure~\ref{fig-simple5})
is a $5$-face,
we might not insert the bold dotted edge, which would result in the face containing only the vertices labelled with $A$ and $C$
in addition to those labelled with $2$, $3$ and $4$.
However, the facial degrees of the vertices labelled with $2$, $3$ and $4$ are $11$, $10$ and $10$, respectively, and
hence the greedy coloring argument would still work.
So, the assumption that all the unconstrained faces around the considered face have size six does not affect the completeness of our arguments.

\begin{figure}
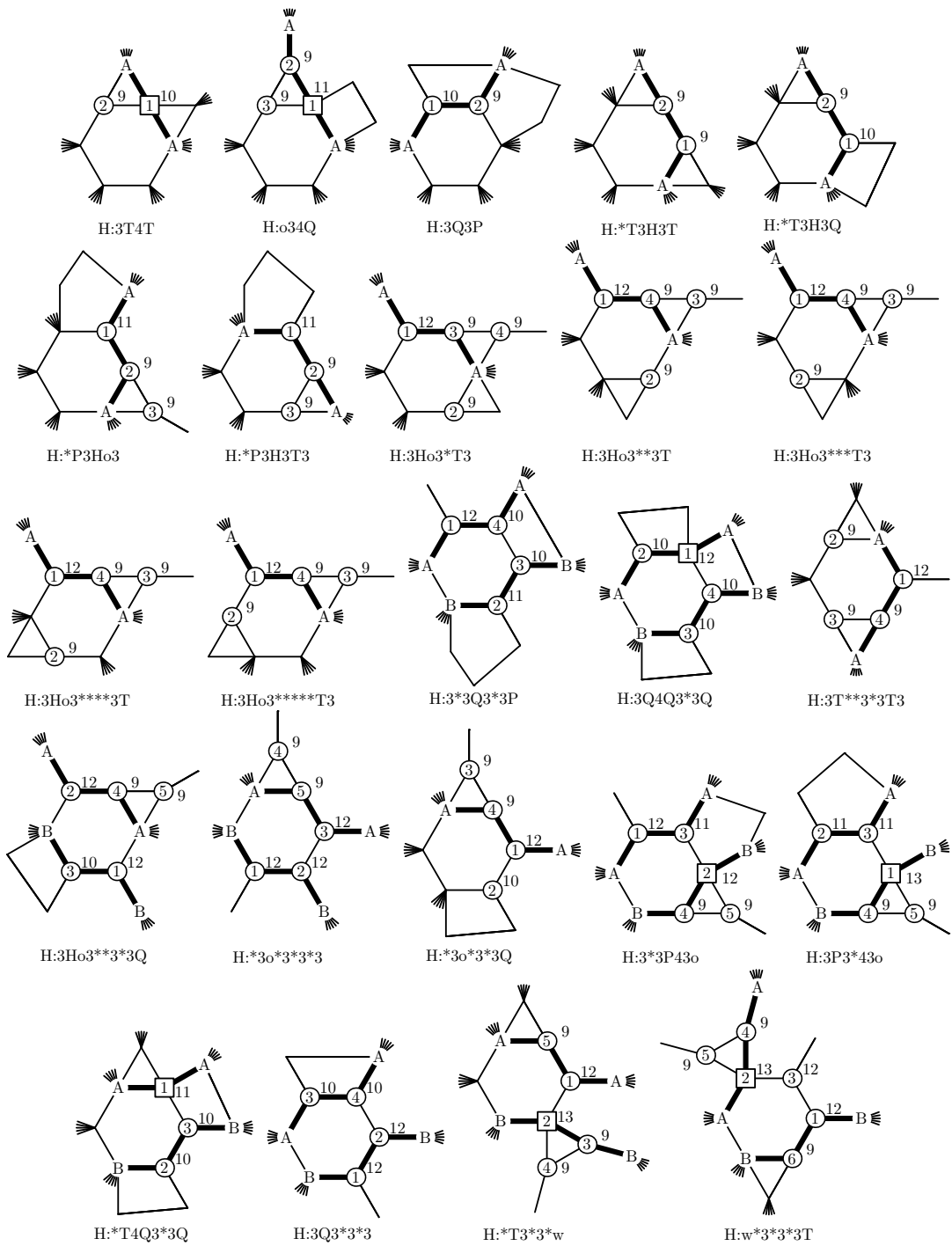

\begin{center}
\epsfbox{cyclic-six.301} \epsfbox{cyclic-six.302} \epsfbox{cyclic-six.303} \epsfbox{cyclic-six.304} \epsfbox{cyclic-six.305}
\epsfbox{cyclic-six.306} \epsfbox{cyclic-six.307} \epsfbox{cyclic-six.308} \epsfbox{cyclic-six.309} \epsfbox{cyclic-six.310}
\epsfbox{cyclic-six.311} \epsfbox{cyclic-six.312} \epsfbox{cyclic-six.313} \epsfbox{cyclic-six.314} \epsfbox{cyclic-six.315}
\epsfbox{cyclic-six.316} \epsfbox{cyclic-six.317} \epsfbox{cyclic-six.318} \epsfbox{cyclic-six.319} \epsfbox{cyclic-six.320}
\epsfbox{cyclic-six.321} \epsfbox{cyclic-six.322} \epsfbox{cyclic-six.323} \epsfbox{cyclic-six.324}
\end{center}
\caption{Configurations related to $6$-faces that are reducible in the simple greedy way---part 1.}
\label{fig-simple6a}
\end{figure}

\begin{figure}
\begin{center}
\epsfbox{cyclic-six.325} \epsfbox{cyclic-six.326} \epsfbox{cyclic-six.327} \epsfbox{cyclic-six.328} \epsfbox{cyclic-six.329}
\epsfbox{cyclic-six.330} \epsfbox{cyclic-six.331} \epsfbox{cyclic-six.332} \epsfbox{cyclic-six.333} \epsfbox{cyclic-six.334} 
\epsfbox{cyclic-six.335} \epsfbox{cyclic-six.336} \epsfbox{cyclic-six.337} \epsfbox{cyclic-six.338} \epsfbox{cyclic-six.339} 
\epsfbox{cyclic-six.340} \epsfbox{cyclic-six.341} \epsfbox{cyclic-six.342} \epsfbox{cyclic-six.343} \epsfbox{cyclic-six.344}
\epsfbox{cyclic-six.345} \epsfbox{cyclic-six.346}
\end{center}
\caption{Configurations related to $6$-faces that are reducible in the simple greedy way---part 2.}
\label{fig-simple6b}
\end{figure}

\begin{figure}
\begin{center}
\epsfbox{cyclic-six.347} \epsfbox{cyclic-six.348} \epsfbox{cyclic-six.349} \epsfbox{cyclic-six.350} \epsfbox{cyclic-six.351}
\epsfbox{cyclic-six.352} \epsfbox{cyclic-six.353} \epsfbox{cyclic-six.354} \epsfbox{cyclic-six.355} \epsfbox{cyclic-six.356}
\epsfbox{cyclic-six.357} \epsfbox{cyclic-six.358} \epsfbox{cyclic-six.359} \epsfbox{cyclic-six.360} \epsfbox{cyclic-six.361}
\epsfbox{cyclic-six.362} \epsfbox{cyclic-six.363} \epsfbox{cyclic-six.364} \epsfbox{cyclic-six.365} \epsfbox{cyclic-six.366}
\epsfbox{cyclic-six.367} \epsfbox{cyclic-six.368} \epsfbox{cyclic-six.369} \epsfbox{cyclic-six.370}
\end{center}
\caption{Configurations related to $6$-faces that are reducible in the simple greedy way---part 3.}
\label{fig-simple6c}
\end{figure}

\begin{figure}
\begin{center}
\epsfbox{cyclic-six.371} \epsfbox{cyclic-six.372} \epsfbox{cyclic-six.373} \epsfbox{cyclic-six.374} \epsfbox{cyclic-six.375}
\epsfbox{cyclic-six.376} \epsfbox{cyclic-six.377} \epsfbox{cyclic-six.378} \epsfbox{cyclic-six.379} \epsfbox{cyclic-six.380}
\epsfbox{cyclic-six.381} \epsfbox{cyclic-six.382} \epsfbox{cyclic-six.383} \epsfbox{cyclic-six.384} \epsfbox{cyclic-six.385}
\epsfbox{cyclic-six.386} \epsfbox{cyclic-six.387} \epsfbox{cyclic-six.388} \epsfbox{cyclic-six.389} \epsfbox{cyclic-six.390}
\epsfbox{cyclic-six.391} \epsfbox{cyclic-six.392}
\end{center}
\caption{Configurations related to $6$-faces that are reducible in the simple greedy way---part 4.}
\label{fig-simple6d}
\end{figure}

\begin{figure}
\begin{center}
\epsfbox{cyclic-six.393} \epsfbox{cyclic-six.394} \epsfbox{cyclic-six.395} \epsfbox{cyclic-six.396} \epsfbox{cyclic-six.397}
\epsfbox{cyclic-six.398} \epsfbox{cyclic-six.399} \epsfbox{cyclic-six.400} \epsfbox{cyclic-six.401} \epsfbox{cyclic-six.402}
\epsfbox{cyclic-six.403} \epsfbox{cyclic-six.404} \epsfbox{cyclic-six.405} \epsfbox{cyclic-six.406} \epsfbox{cyclic-six.407}
\epsfbox{cyclic-six.408} \epsfbox{cyclic-six.409} \epsfbox{cyclic-six.410}
\end{center}
\caption{Configurations related to $6$-faces that are reducible in the simple greedy way---part 5.}
\label{fig-simple6e}
\end{figure}

\begin{figure}
\begin{center}
\epsfbox{cyclic-six.411} \epsfbox{cyclic-six.412} \epsfbox{cyclic-six.413} \epsfbox{cyclic-six.414} \epsfbox{cyclic-six.415}
\epsfbox{cyclic-six.416} \epsfbox{cyclic-six.417} \epsfbox{cyclic-six.418} \epsfbox{cyclic-six.419} \epsfbox{cyclic-six.420}
\epsfbox{cyclic-six.421} \epsfbox{cyclic-six.422} \epsfbox{cyclic-six.423} \epsfbox{cyclic-six.424} \epsfbox{cyclic-six.425}
\epsfbox{cyclic-six.426} \epsfbox{cyclic-six.427} \epsfbox{cyclic-six.428} \epsfbox{cyclic-six.429} \epsfbox{cyclic-six.430}
\epsfbox{cyclic-six.431} \epsfbox{cyclic-six.432}
\end{center}
\caption{Configurations related to $6$-faces that are reducible in the simple greedy way---part 6.}
\label{fig-simple6f}
\end{figure}

\begin{figure}
\begin{center}
\epsfbox{cyclic-six.433} \epsfbox{cyclic-six.434} \epsfbox{cyclic-six.435} \epsfbox{cyclic-six.436} \epsfbox{cyclic-six.437}
\epsfbox{cyclic-six.438} \epsfbox{cyclic-six.439} \epsfbox{cyclic-six.440} \epsfbox{cyclic-six.441} \epsfbox{cyclic-six.442}
\epsfbox{cyclic-six.443} \epsfbox{cyclic-six.444} \epsfbox{cyclic-six.445} \epsfbox{cyclic-six.446} \epsfbox{cyclic-six.447}
\epsfbox{cyclic-six.448} \epsfbox{cyclic-six.449} \epsfbox{cyclic-six.450} \epsfbox{cyclic-six.451} \epsfbox{cyclic-six.452}
\epsfbox{cyclic-six.453}
\end{center}
\caption{Configurations related to $6$-faces that are reducible in the simple greedy way---part 7.}
\label{fig-simple6g}
\end{figure}

\begin{figure}
\begin{center}
\epsfbox{cyclic-six.454}
\epsfbox{cyclic-six.455} \epsfbox{cyclic-six.456} \epsfbox{cyclic-six.457} \epsfbox{cyclic-six.458} \epsfbox{cyclic-six.459}
\epsfbox{cyclic-six.460} \epsfbox{cyclic-six.461} \epsfbox{cyclic-six.462} \epsfbox{cyclic-six.463} \epsfbox{cyclic-six.464}
\epsfbox{cyclic-six.465} \epsfbox{cyclic-six.466} \epsfbox{cyclic-six.467} \epsfbox{cyclic-six.468} \epsfbox{cyclic-six.469}
\end{center}
\caption{Configurations related to $6$-faces that are reducible in the simple greedy way---part 8.}
\label{fig-simple6h}
\end{figure}

\begin{figure}
\begin{center}
\epsfbox{cyclic-six.470}
\epsfbox{cyclic-six.471} \epsfbox{cyclic-six.472} \epsfbox{cyclic-six.473} \epsfbox{cyclic-six.474} \epsfbox{cyclic-six.475}
\epsfbox{cyclic-six.476} \epsfbox{cyclic-six.477} \epsfbox{cyclic-six.478} \epsfbox{cyclic-six.479} \epsfbox{cyclic-six.480}
\epsfbox{cyclic-six.481} \epsfbox{cyclic-six.482} \epsfbox{cyclic-six.483} \epsfbox{cyclic-six.484} \epsfbox{cyclic-six.485}
\epsfbox{cyclic-six.486}
\end{center}
\caption{Configurations related to $6$-faces that are reducible in the simple greedy way---part 9.}
\label{fig-simple6i}
\end{figure}

\subsection{List coloring argument}

The reducibility of four configurations in our proof was established using arguments involving list coloring.
In list coloring, each vertex of a graph is assigned a list of available colors and
a proper vertex coloring such that each vertex receives a color from its list is sought (a proper vertex coloring is
a coloring such that no two adjacent vertices receive the same color).
To demonstrate the concept,
we start with a (very simple) auxiliary lemma, which is used in most of our reductions.

\begin{lemma}
\label{lm-K4}
Let $G$ be a graph with vertices $\alpha$, $\beta$, $\gamma$ and $\delta$ such that
all the pairs of vertices are adjacent except for the pair $\alpha$ and $\beta$.
Suppose that each of the vertices $\alpha$ and $\beta$ is assigned a list of two colors and
each of the vertices $\gamma$ and $\delta$ is assigned a list of three colors.
The vertices of the graph $G$ can be properly colored such that each vertex receives a color from its list.
\end{lemma}

\begin{proof}
We distinguish two cases.
If there is a color contained in the lists of both vertices $\alpha$ and $\beta$,
color both vertices $\alpha$ and $\beta$ with this color and
then color the vertices $\gamma$ and $\delta$ (in this order) with any colors from their lists not assigned to any of their neighbors.
On the other hand, if the lists of the vertices $\alpha$ and $\beta$ are disjoint,
their union contains four colors and thus it contains a color not in the list of the vertex $\delta$.
Let $x$ be this color. By symmetry, we can assume that $x$ is contained in the list of $\alpha$.
We color the vertex $\alpha$ by $x$, the vertex $\beta$ by any color from its list,
the vertex $\gamma$ by any color from its list different from the colors of $\alpha$ and $\beta$, and
finally the vertex $\delta$ by any color from its list different from the colors of $\beta$ and $\gamma$.
Since $x$ is not contained in the list of $\delta$, the color assigned to $\delta$ is different from $x$ and
the coloring that we have obtained is proper.
\end{proof}

Lemma~\ref{lm-K4} is used to establish the reducibility of the configurations in the next lemma.

\begin{lemma}
\label{lm-list-K4}
The configurations \conf{H:3T**oQ3}, \conf{H:*T3***oQ3} and \conf{H:3T****oQ3} are reducible.
\end{lemma}

\begin{figure}
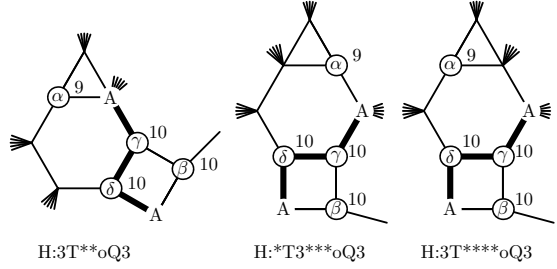

\begin{center}
\epsfbox{cyclic-six.104}
\epsfbox{cyclic-six.102}
\epsfbox{cyclic-six.103}
\end{center}
\caption{The reducible configurations from Lemma~\ref{lm-list-K4}.}
\label{fig-list-K4}
\end{figure}

\begin{proof}
The configurations from the statement of the lemma are depicted in Figure~\ref{lm-list-K4}.
We follow the notation from Subsection~\ref{sub-greedy}.
Suppose that a minimal counterexample contains one of the configurations.
As in Subsection~\ref{sub-greedy}, we contract the subgraphs depicted in bold, obtain a coloring of the new graph and
assign the colors to the vertices labelled with $A$ based on the coloring we obtained (note that the pair of such vertices
is not facially adjacent by Lemma~\ref{lm-dist3}).
We next uncolor the vertices labelled with $\alpha$, $\beta$, $\gamma$ and $\delta$ (if they are colored).
The facial degrees and the facial adjacencies to the pairs of vertices with the same color
yield that
there are at least two colors not assigned to the facial neighbors of $\alpha$,
at least two colors not assigned to the facial neighbors of $\beta$,
at least three colors not assigned to the facial neighbors of $\gamma$ and
at least three colors not assigned to the facial neighbors of $\delta$.
Since the vertices $\alpha$ and $\beta$ are not facially adjacent by Lemma~\ref{lm-dist3},
we can complete the coloring to a cyclic coloring by Lemma~\ref{lm-K4}.
\end{proof}

We finish this subsection with a more involved list coloring argument.
Since this argument applies only to the configuration considered in the next lemma,
we present the argument in the specific setting of the considered configuration only.

\begin{lemma}
\label{lm-list-special}
The configuration \conf{H:3Q4Po*3P} is reducible.
\end{lemma}

\begin{figure}
\begin{center}
\epsfbox{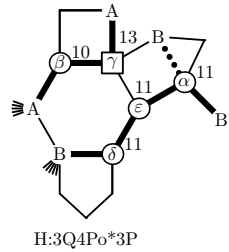}
\end{center}
\caption{The configuration from Lemma~\ref{lm-list-special}.}
\label{fig-list-special}
\end{figure}

\begin{proof}
Suppose that a minimal counterexample $G$ contains the configuration \conf{H:3Q4Po*3P}.
Add the dotted edge depicted in Figure~\ref{fig-list-special} and
contract the two subgraphs formed by bold edges.
By the minimality of $G$, the obtained graph has a cyclic coloring with at most nine colors.
All the vertices keep their colors and the vertices labelled with $A$ and $B$ get the colors of
the vertices corresponding to the contracted subgraphs.
Note that the vertices to be colored with the same color are not facially adjacent by Lemmas~\ref{lm-dist3} and~\ref{lm-dist4}.
We are now left to color the vertices $\alpha$, $\beta$, $\gamma$, $\delta$ and $\varepsilon$.
Observe that all the pairs of these five vertices are facially adjacent except for the pair $\alpha$ and $\beta$,
which is not facially adjacent by Lemma~\ref{lm-dist3}.

From the facial degrees and the facial adjacencies to the vertices with same color, we derive that
there are at least two colors available for each of the vertices $\alpha$ and $\gamma$,
at least three colors available for each of the vertices $\beta$ and $\delta$ and
at least four colors available for the vertex $\varepsilon$.
Let $Z$ be a set formed by four colors available for $\varepsilon$.

If there is a color that can be assigned to both $\alpha$ and $\beta$,
then we color both $\alpha$ and $\beta$ with this color and
the remaining vertices in the order $\gamma$, $\delta$ and $\varepsilon$.
Assume now that there is no color available to both $\alpha$ and $\beta$.
Since there are at least five colors in total available to $\alpha$ or $\beta$,
one of these colors, say $x$, is not contained in the set $Z$.

If $x$ is available for the vertex $\alpha$, we color $\alpha$ with this color and
color the remaining vertices in the order $\gamma$, $\delta$, $\beta$ and $\varepsilon$.
So, we can assume that the color $x$ is available for the vertex $\beta$.
We start with coloring the vertices $\alpha$, $\gamma$ and $\delta$ (in this order) with arbitrary available colors.
If neither $\gamma$ nor $\delta$ is colored with the color $x$, we color $\beta$ with $x$.
Otherwise, we color $\beta$ with an arbitrary color that is available for $\beta$ and that has not been assigned to $\gamma$ or $\delta$.
In both cases, the vertex $\varepsilon$ has a facial neighbor colored with $x$ and we can complete the coloring to a cyclic coloring of $G$.
\end{proof}

\subsection{Special arguments}

In this subsection, we establish reducibility of three additional configurations using ad hoc arguments.

\begin{lemma}
\label{lm-o3o}
The configuration \conf{H:o3o} is reducible.
\end{lemma}

\begin{proof}
Let $G$ be a minimal counterexample and
let $uvw$ be a $3$-face of $G$ such that all the three vertices $u$, $v$ and $w$ are $3$-vertices.
Note that the facial degree of all the three vertices $u$, $v$ and $w$ in $G$ is $9$.
Contract the triangle $uvw$ to a single vertex, color the obtained graph $G'$ by the minimality of $G$, and
assign the vertices of $G$ except for $u$, $v$ and $w$ the colors they are assigned in $G'$.

If one of the vertices $u$, $v$ and $w$, say $w$, is facially adjacent to two vertices of the same color,
we can color the three vertices in the order $u$, $v$ and $w$ greedily.
So, we assume that none of the vertices $u$, $v$ and $w$ are facially adjacent to two vertices of the same color.
Let $X_{uv}$ be the colors of the two vertices incident with the $6$-face containing the edge $uv$ that are not the neighbors of $u$ or $v$.
We use $X_{uw}$ and $X_{vw}$ in the analogous way with respect to the other two $6$-faces sharing the edges with the $3$-face.
The assumption that none of the vertices $u$, $v$ and $w$ are facially adjacent to two vertices of the same color
implies that the sets $X_{uv}$, $X_{uw}$ and $X_{vw}$ are disjoint and they do not contain a color of any neighbor of the vertices $u$, $v$ and $w$.
We can now complete the coloring by assigning the vertex $u$ an arbitrary color from $X_{vw}$,
the vertex $v$ an arbitrary color from $X_{uw}$ and
the vertex $w$ an arbitrary color from $X_{uv}$.
\end{proof}

\begin{lemma}
\label{lm-special1}
The configurations \conf{H:3*3T4Po*3} and \conf{H:3Po*4T3*3} are reducible.
\end{lemma}

\begin{figure}
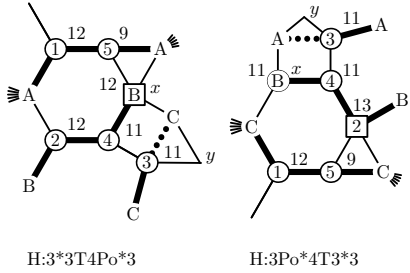

\begin{center}
\epsfbox{cyclic-six.201}
\epsfbox{cyclic-six.202}
\end{center}
\caption{The configurations from Lemma~\ref{lm-special1}.}
\label{fig-special1}
\end{figure}

\begin{proof}
We proceed in a way similar to the simple greedy reductions from Subsection~\ref{sub-greedy}.
We consider a minimal counterexample $G$ that contains one of the configurations \conf{H:3*3T4Po*3} and \conf{H:3Po*4T3*3},
which are depicted in Figure~\ref{fig-special1}.
We start with inserting the dotted edge and contracting the three subgraphs formed by bold edges.
By the minimality of $G$, we obtain a cyclic coloring of the new graph,
which gives the coloring to all the vertices of $G$ except the ones contained in the contracted subgraphs.
The vertices labelled with $A$, $B$ and $C$ get the colors of the vertices corresponding
to the contracted subgraphs (each of the three pairs of these vertices is not facially adjacent by Lemma~\ref{lm-dist3}).
However, the vertices $x$ and $y$ may have the same color.

If the vertices $x$ and $y$ have different colors,
we color the remaining vertices greedily in the order given by the numbering in Figure~\ref{fig-special1}.
If the vertices $x$ and $y$ have the same color, we uncolor the vertex $x$.
Note that the vertices labelled by $3$ and $4$ are still facially adjacent to two pairs of vertices with the same color (one of the pairs contains the vertex $y$).
We now color the six uncolored vertices greedily in the order given by the numbering in Figure~\ref{fig-special1}
with the vertex $x$ being colored between the vertices labelled by $2$ and $3$.
\end{proof}

\section{Discharging rules}
\label{sec-rules}

The discharging rules are listed in Tables~\ref{tab-Trules}, \ref{tab-Prules} and \ref{tab-Hrules} using the encoding we now describe.
There are three basic types of discharging rules: $T$-rules, $P$-rules and $H$-rules.
The $T$-rules are described by strings of nine characters starting with \conf{T:}.
If a $6$-face $v_1v_2\cdots v_6$ matches the description given by the rule,
i.e., the vertices $v_i$, $i\in\{1,2,3,4\}$, correspond to the $(2i+1)$-th characters and
the faces sharing the edges $v_iv_{i+1}$, $i\in\{1,2,3\}$, correspond to the $(2i+2)$-th characters,
then the $6$-face $v_1v_2\cdots v_6$ sends the prescribed amount of charge to the face sharing the edge $v_2v_3$.
The face sharing the edge $v_2v_3$ with the $6$-face will always be a $3$-face. Moreover,
at most one of the $T$-rules will apply to any pair of a $6$-face and a $3$-face sharing an edge.

\begin{table}
\begin{center}
\begin{tabular}{crcrcr}
\conf{T:3H3x3Hx} & 10/60 & \conf{T:3H3x4P*} & 10/60 & \conf{T:*P4O4P*} & 32/60 \\
\conf{T:3Hot4P*} &  1/60 & \conf{T:xH3x4P*} & 13/60 & \conf{T:*P4t4H*} & 26/60 \\
\conf{T:xHot4P*} & 20/60 & \conf{T:3H3x4H*} &  8/60 & \conf{T:*P4O4H*} & 31/60 \\
\conf{T:3HoO4P*} & 20/60 & \conf{T:xH3x4H*} & 10/60 & \conf{T:*H4t4H*} & 26/60 \\
\conf{T:xHoO4P*} & 29/60 & \conf{T:3H3x+**} & 14/60 & \conf{T:*H4O4H*} & 30/60 \\
\conf{T:3Hot4H*} & 10/60 & \conf{T:xH3x+**} & 12/60 & \conf{T:*Q4t+**} & 22/60 \\
\conf{T:xHot4H*} & 20/60 & \conf{T:*Q4t4Q*} &  8/60 & \conf{T:*Q4O+**} & 24/60 \\
\conf{T:*HoO4H*} & 20/60 & \conf{T:*Q4O4Q*} & 16/60 & \conf{T:*P4t+**} & 31/60 \\
\conf{T:3Hot+**} & 20/60 & \conf{T:*Q4t4P*} & 17/60 & \conf{T:*P4O+**} & 32/60 \\ 
\conf{T:xHot+**} & 30/60 & \conf{T:*Q4O4P*} & 24/60 & \conf{T:*H43+**} & 31/60 \\
\conf{T:*HoO+**} & 30/60 & \conf{T:*Q4t4H*} & 17/60 & \conf{T:**+t+**} & 36/60 \\
\conf{T:3H3x4Q*} & 22/60 & \conf{T:*Q4O4H*} & 23/60 & \conf{T:**+O+**} & 32/60 \\
\conf{T:xH3x4Q*} & 26/60 & \conf{T:*P4t4P*} & 26/60 & \conf{T:**xxx**} & 20/60 \\ 
\end{tabular}
\end{center}
\caption{The $T$-rules.}
\label{tab-Trules}
\end{table}

\begin{table}
\begin{center}
\begin{tabular}{crcrcrcr}
\conf{P:3Q3H*} & 40/60 & \conf{P:xPoPx} & 40/60 & \conf{P:xPtHx} & 20/60 & \conf{P:3H3H+} & 14/60 \\
\conf{P:xQ3H*} & 20/60 & \conf{P:3PtH3} & 12/60 & \conf{P:4PoHx} & 18/60 & \conf{P:4H3H4} & 20/60 \\
\conf{P:3PtP3} & 12/60 & \conf{P:3PtH4} & 18/60 & \conf{P:+PtH3} & 12/60 & \conf{P:4H3H+} & 26/60 \\
\conf{P:3PtPx} & 10/60 & \conf{P:3PtH+} & 20/60 & \conf{P:+PoHx} & 24/60 & \conf{P:+H3H+} & 32/60 \\ 
\conf{P:xPtPx} & 20/60 & \conf{P:*PoH3} & 18/60 & \conf{P:3HtH3} & 12/60 & \conf{P:**+**} & -12/60 \\ 
\conf{P:3PoP3} & 20/60 & \conf{P:3PoHx} & 20/60 & \conf{P:3HoH3} & 20/60 & \conf{P:**u**} & 4/60 \\   
\conf{P:3PoPx} & 24/60 & \conf{P:4PtH3} & 18/60 & \conf{P:3H3H4} & 16/60 & \conf{P:**w**} & 20/60 \\  
\end{tabular}
\end{center}
\caption{The $P$-rules.}
\label{tab-Prules}
\end{table}

\begin{table}
\begin{center}
\begin{tabular}{crcrcrcr}
\conf{H:3TtH3} & 20/60 & \conf{H:3QtH*} & 24/60 & \conf{H:+P3P+} & 20/60 & \conf{H:*H3H*} & 20/60 \\
\conf{H:3TtH4} & 30/60 & \conf{H:3QoH*} & 30/60 & \conf{H:3P3H3} & 20/60 & \conf{H:*T5T*} & -24/60 \\
\conf{H:3TtH+} & 36/60 & \conf{H:xQtH*} & 30/60 & \conf{H:3PtHx} & 22/60 & \conf{H:*T6T*} & -40/60 \\
\conf{H:xTtH*} & 30/60 & \conf{H:xQoH*} & 36/60 & \conf{H:3PoHx} & 24/60 & \conf{H:*T+Q*} & -24/60 \\
\conf{H:xToH3} & 40/60 & \conf{H:3P3P3} & 24/60 & \conf{H:4PtH*} & 20/60 & \conf{H:*T+P*} & -18/60 \\
\conf{H:xToH4} & 30/60 & \conf{H:3PtPx} & 24/60 & \conf{H:4PoH3} & 24/60 & \conf{H:*T+H*} & -18/60 \\
\conf{H:xToH+} & 24/60 & \conf{H:3PoPx} & 28/60 & \conf{H:4PoHx} & 22/60 & \conf{H:*F+F*} & -12/60 \\
\conf{H:*QtP*} & 40/60 & \conf{H:4PtPx} & 20/60 & \conf{H:+PoH*} & 26/60 & \conf{H:**u**} & 7/60 \\   
\conf{H:*QoP*} & 20/60 & \conf{H:4PoPx} & 24/60 & \conf{H:+PtH*} & 14/60 & \conf{H:**w**} & 20/60 \\   
\end{tabular}
\end{center}
\caption{The $H$-rules.}
\label{tab-Hrules}
\end{table}

The $P$-rules and $H$-rules are described by strings of seven characters starting with \conf{P:} and \conf{H:}.
If a face $f$ matches the description given by the rule,
then the face $f$ sends the prescribed amount of charge to the second vertex (the one corresponding to the fifth character).
In particular, the $P$-rules apply to $5$-faces and the $H$-rules to $6$-faces.
If the prescribed amount of charge is negative (this happens in one of the $P$-rules and in four of the $H$-rules),
the face $f$ receives the corresponding amount of charge.
Finally, if the charge sent by $f$ goes to a $4$-vertex,
the $4$-vertex resends all of the received charge to the $3$-face incident with it (this is the case for the last two $P$-rules and
the last two $H$-rules).
As in the case of $T$-rules,
at most one of the $P$-rules and $H$-rules applies to any pair of a face and an incident vertex.

In the next two lemmas, we analyze the final amount of charge of vertices and $3$-faces.

\begin{lemma}
\label{lm-ver3}
Let $G$ be a $2$-connected plane graph with maximum face size six and $v$ a vertex of $G$.
If the minimum facial degree of $G$ is at least nine and $G$ does not contain a $3$-face incident with three $3$-vertices,
then the final amount of charge of $v$ is non-negative.
\end{lemma}

\begin{figure}
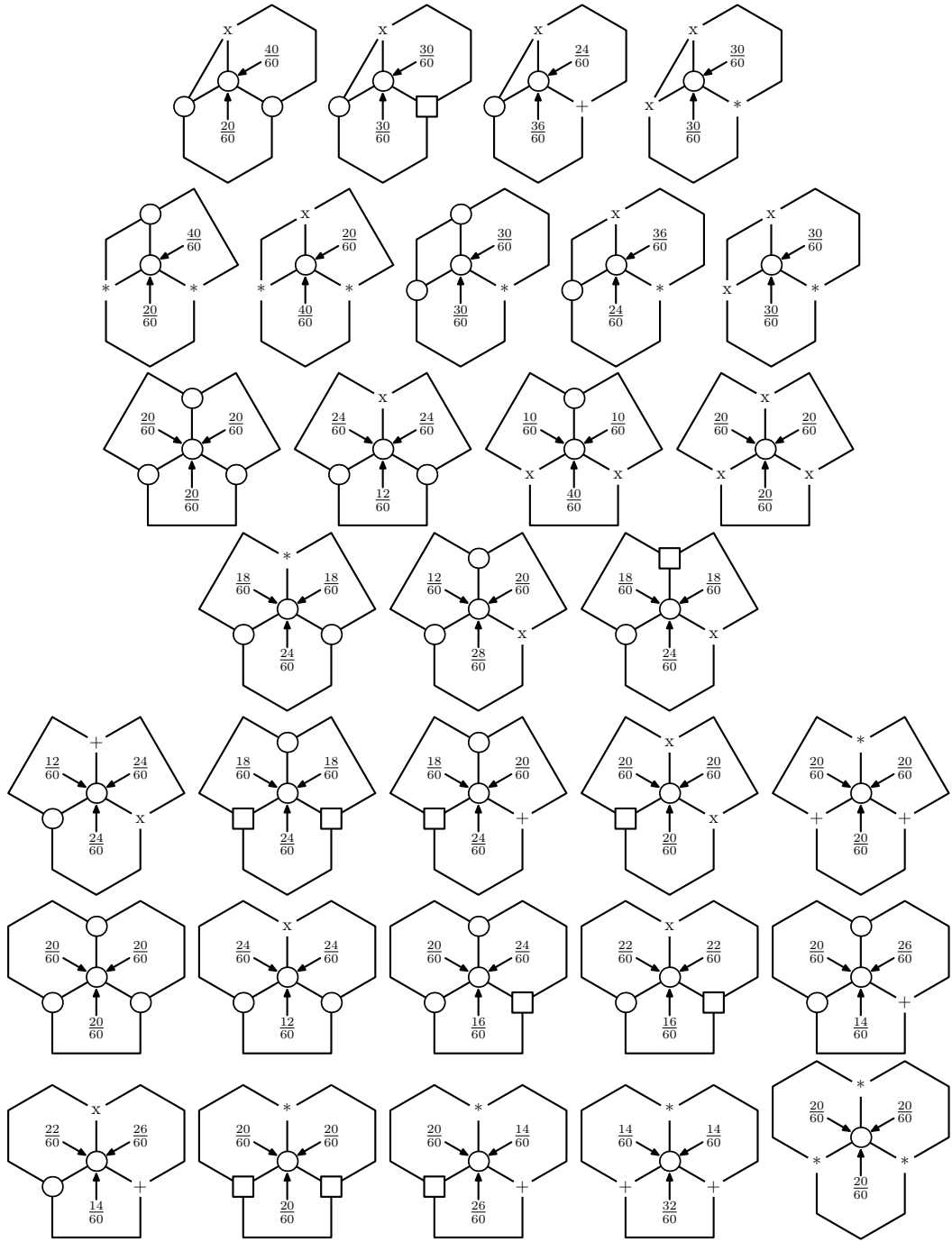

\begin{center}
\epsfbox{cyclic-six.2002}
\epsfbox{cyclic-six.2003}
\epsfbox{cyclic-six.2004}
\epsfbox{cyclic-six.2005}\\
\epsfbox{cyclic-six.2011}
\epsfbox{cyclic-six.2013}
\epsfbox{cyclic-six.2021}
\epsfbox{cyclic-six.2022}
\epsfbox{cyclic-six.2023}\\
\epsfbox{cyclic-six.2031}
\epsfbox{cyclic-six.2032}
\epsfbox{cyclic-six.2033}
\epsfbox{cyclic-six.2034}\\
\epsfbox{cyclic-six.2041}
\epsfbox{cyclic-six.2042}
\epsfbox{cyclic-six.2043}\\
\epsfbox{cyclic-six.2044}
\epsfbox{cyclic-six.2045}
\epsfbox{cyclic-six.2046}
\epsfbox{cyclic-six.2047}
\epsfbox{cyclic-six.2048}\\
\epsfbox{cyclic-six.2051}
\epsfbox{cyclic-six.2052}
\epsfbox{cyclic-six.2053}
\epsfbox{cyclic-six.2054}
\epsfbox{cyclic-six.2055}
\epsfbox{cyclic-six.2056}
\epsfbox{cyclic-six.2057}
\epsfbox{cyclic-six.2058}
\epsfbox{cyclic-six.2059}
\epsfbox{cyclic-six.2061}
\end{center}
\caption{Charge received by $3$-vertices. The degrees of vertices are encoded using the notation for configurations.}
\label{fig-3ver}
\end{figure}

\begin{proof}
Since the minimum facial degree of $G$ is at least nine, the minimum degree of $G$ is at least three.
If $v$ is a $3$-vertex, then one of the cases depicted in Figure~\ref{fig-3ver} holds and
the vertex $v$ receives at least one unit of charge in total from the incident $\ge 5$-faces.
Since $4$-vertices do not send out or receive any charge except for that they immediately resend to $3$-faces,
we assume from now on that $v$ is a $\ge 5$-vertex.

Let $t$ be the number of $3$-faces incident with $v$, $q$ the number of $4$-faces and $p$ be the number of $\ge 5$-faces.
Suppose that $v$ is a $5$-vertex.
The vertex $v$ sends $12/60$ to each incident $\ge 5$-face $f$ by the rules \conf{P:**+**} and \conf{H:*F+F*}
unless one of the two other faces incident with $v$ that shares an edge with $f$ is a $3$-face.
The amount of charge sent is increased by $12/60$ for each $6$-face sharing an edge with a $3$-face and 
a $\le 4$-face incident with $v$ (see the rules \conf{H:*T5T*} and \conf{H:*T+Q*}), and
is increased by $6/60$ for each $6$-face sharing an edge with a $3$-face and $\ge 5$-face
incident with $v$ (see the rules \conf{H:*T+P*} and \conf{H:*T+H*}).
So, each $3$-face incident with $v$ increases the amount of charge sent from $v$ to a $6$-face that shares an edge with it by $6/60$ units and
each $4$-face incident with $v$ can increase the amount of charge sent from $v$ to a $6$-face that shares an edge with it by $6/60$ units (this happens only if the other face incident with $v$ that shares an edge with the $6$-face is a $3$-face).
Since each $3$-face and $4$-face shares an edge with at most two faces incident with $v$,
we conclude that the $5$-vertex $v$ sends out at most $(12t+12q+12p)/60\le 1$ unit of charge and
its final amount of charge is non-negative.

Suppose that $v$ is a $d$-vertex, $d\ge 6$.
The calculation is the same except that each $6$-face sharing edges with two $3$-faces incident with $v$
gets $40/60$ units of charge from $v$ instead of $24/60$ units (the rule \conf{H:*T6T*} applies instead of \conf{H:*T5T*}).
Hence, the additional amount of charge sent out can be up to $28/60$ units per incident $3$-face instead of $12/60$ units as in the previous case.
This yields that $v$ sends out at most $(28t+12q+12p)/60$ units of charge.
Since a $3$-face can share an edge only with a $6$-face, we get that $t\le p$.
Consequently, the $d$-vertex $v$ sends out at most
$$\frac{28t+12q+12p}{60}\le \frac{20t+12q+20p}{60}\le \frac{t+q+p}{3}=\frac{d}{3}\le d-4$$
units of charge and its final amount of charge is non-negative.
\end{proof}

\begin{lemma}
\label{lm-fac3}
Let $G$ be a $2$-connected plane graph with maximum face size six and let $v_1v_2v_3$ be a $3$-face of $G$ that does not share an edge with a $\le 5$-face.
If the minimum facial degree of $G$ at least nine and $G$ does not contain \conf{H:o3o}, \conf{H:3T4T} or $\conf{H:o34Q}$,
then the face $v_1v_2v_3$ receives at least one unit of charge using the $T$-rules, $P$-rules and $H$-rules.
\end{lemma}

\begin{figure}
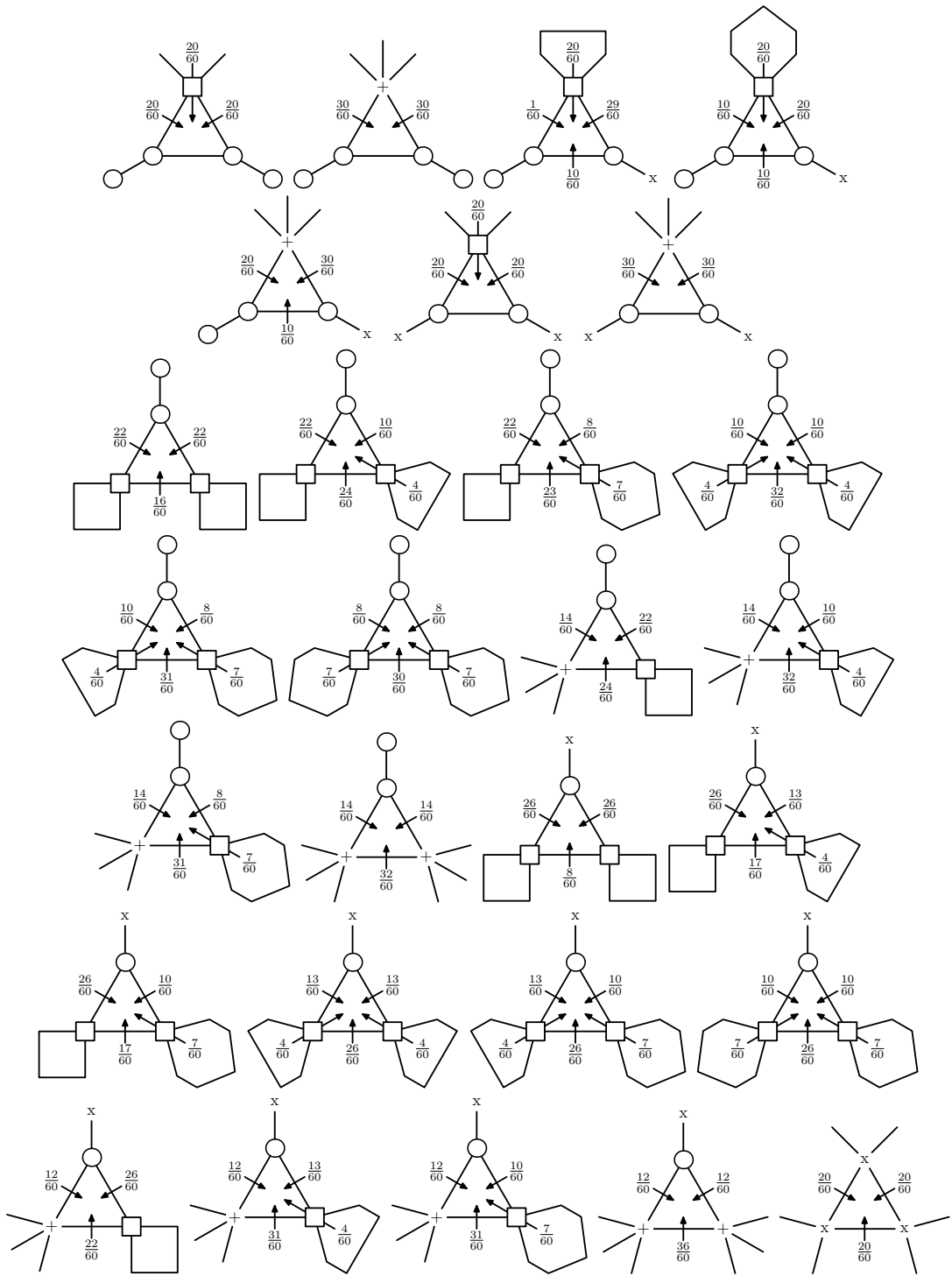

\begin{center}
\epsfbox{cyclic-six.3011}
\epsfbox{cyclic-six.3012}
\epsfbox{cyclic-six.3013}
\epsfbox{cyclic-six.3014}\\
\epsfbox{cyclic-six.3015}
\epsfbox{cyclic-six.3016}
\epsfbox{cyclic-six.3017}\\
\epsfbox{cyclic-six.3021}
\epsfbox{cyclic-six.3022}
\epsfbox{cyclic-six.3023}
\epsfbox{cyclic-six.3024}
\epsfbox{cyclic-six.3025}
\epsfbox{cyclic-six.3026}
\epsfbox{cyclic-six.3027}
\epsfbox{cyclic-six.3028}
\epsfbox{cyclic-six.3029}
\epsfbox{cyclic-six.3030}
\epsfbox{cyclic-six.3041}
\epsfbox{cyclic-six.3042}
\epsfbox{cyclic-six.3043}
\epsfbox{cyclic-six.3044}
\epsfbox{cyclic-six.3045}
\epsfbox{cyclic-six.3046}
\epsfbox{cyclic-six.3047}
\epsfbox{cyclic-six.3048}
\epsfbox{cyclic-six.3049}
\epsfbox{cyclic-six.3050}
\epsfbox{cyclic-six.3061}
\end{center}
\caption{Charge received by $3$-faces. The degrees of vertices are encoded using the notation for configurations.}
\label{fig-3fac}
\end{figure}

\begin{proof}
All possible configurations around $3$-faces in a graph satisfying the assumption of the lemma are depicted in Figure~\ref{fig-3fac}.
The picture also contains the amounts of charge received by such $3$-faces and
it can be verified that the final amount of charge of the $3$-face is always non-negative.
\end{proof}

The analysis of the final amount of charge of $5$-faces and $6$-faces turned out to be too complex.
So, we had to verify that the final amount of charge of such faces is non-negative with the assistance of a computer.
We have prepared three computer programs and
we have made one of them available at \url;
the program is also available on arXiv as an ancillary file.

\begin{lemma}
\label{lm-fac56}
Let $G$ be a $2$-connected plane graph with maximum face size six and let $f$ be a $d$-face of $G$, $d\in\{5,6\}$.
If $G$ contains none of the reducible configurations,
its minimum facial degree is at least nine and
there is no $\le 5$-face sharing an edge with a $3$-face,
then the difference between the amount of charge sent out by $f$ and received by it is at most $d-4$ units.
\end{lemma}

Lemmas~\ref{lm-ver3}, \ref{lm-fac3} and \ref{lm-fac56} together with the absence of any of the reducible configurations in a minimal counterexample
exclude the existence of a minimal counterexample for Theorem~\ref{thm-main}; this finishes the proof of Theorem~\ref{thm-main}.

\section*{Acknowledgement}

The authors would like to thank Jakub Slia\v can for discussions related to the computer assisted proofs contained in this paper.

\end{document}